\documentclass[11pt]{article}

\usepackage{amsmath,amssymb}
\usepackage{amsthm}
\usepackage{amstext}
\usepackage[numeric,initials,nobysame]{amsrefs}
\usepackage{verbatim}
\def\ver{  }
\newtheorem{theorem}{Theorem}[section]

\newtheorem{lemma}[theorem]{Lemma}

\newtheorem{proposition}[theorem]{Proposition}

\newtheorem{remark}[theorem]{Remark}
\numberwithin{equation}{section}

\def\eps{\varepsilon}

\def\phi{\varphi}

\def\qed{{\hfill $\square$ \bigskip}}

\def\diam{{\mathop {{\rm diam\, }}}}

\def\ignore#1{}  

\newcommand{\deq}{\stackrel{\scriptscriptstyle\triangle}{=}}
\newcommand{\N}{\mathbb{N}}
\newcommand{\E}{\mathbb{E}}
\newcommand{\tcov}{t_{\mathrm{cov}}}
\newcommand{\taucov}{\tau_{\mathrm{cov}}}
\renewcommand{\P}{\mathbb{P}}

\newcommand{\res}{R_{{\bf{eff}}}}
\newcommand{\ball}{B_{{\bf{eff}}}}
\newcommand{\be}{\begin{eqnarray}}
\newcommand{\ee}{\end{eqnarray}}
\newcommand{\K}{\mathcal{K}}
\newcommand{\GC}{{\mathcal{C}_1}} 
\newcommand{\one}{\mathbf{1}}
\newcommand{\tGC}{{\tilde{\mathcal{C}}_1}}
\newcommand{\C}{{\mathcal{C}}}
\newcommand{\Z}{\mathbb{Z}}
\newcommand{\bcr}{B}
\renewcommand{\diam}{{\rm diam}}
\newcommand{\diameff}{{\rm diam}_{{\bf{eff}}}}

\renewcommand{\and}{\hbox{ {\rm and} }}

\begin{document}

\title{\bf The evolution of the cover time}

\author{
Martin T. Barlow\footnote{Department of mathematics, University of British Columbia, Research partially supported by NSERC
(Canada) and the Peter Wall Institute of Advanced Studies. Email: barlow@math.ubc.ca}, Jian
Ding\footnote{Department of statistics, University of California at Berkeley, partially supported by Microsoft Research. Email: jding@berkeley.edu}, Asaf Nachmias \footnote{Department of Mathematics, Massachusetts Institute of Technology. Email: asafnach@math.mit.edu}, Yuval Peres\footnote{Microsoft Research. Email: peres@microsoft.com}}
\maketitle

\ver
\begin{abstract}
The cover time of a graph is a celebrated example of a parameter
that is easy to approximate using a randomized algorithm, but for
which no constant factor deterministic polynomial time approximation
is known. A breakthrough due to Kahn, Kim, Lov{\'a}sz and
Vu~\cite{KKLV} yielded a $(\log\log n)^2$ polynomial time
approximation. We refine the  upper bound of \cite{KKLV}, and show
that the resulting bound is sharp and explicitly computable in
random graphs. Cooper and Frieze showed that the cover time of the
largest component of the Erd\H{o}s-R\'enyi random graph $G(n,c/n)$
in the supercritical regime with $c>1$ fixed, is asymptotic to
$\varphi(c) n \log^2 n$, where   $\varphi(c) \to 1$ as $c \downarrow
1$. However, our new bound implies that the cover time for the
critical Erd\H{o}s-R\'enyi random graph $G(n,1/n)$ has order $n$,
and shows how the cover time evolves from the critical window to the
supercritical phase. Our general estimate also yields  the order of
the cover time for a variety of other concrete graphs, including
critical percolation clusters on the Hamming hypercube $\{0,1\}^n$,
on high-girth expanders, and on  tori $\Z_n^d$ for fixed large $d$. This approach also gives a simpler proof of a result of Aldous \cite{Aldous} that the cover time of a uniform labeled tree on $k$ vertices is of order $k^{3/2}$. 
For the graphs we consider, our results show that the {\em blanket\/} time, introduced by Winkler and Zuckerman \cite{WZ}, is within a constant factor of the cover time. Finally, we prove that for any connected graph, adding an edge can increase the cover time by at most a factor of 4.
\end{abstract}

\newpage

\section{Introduction}

The cover time $\tcov(G)$ of a graph $G$ is the expected number of steps a simple random walk takes to
visit every vertex of the graph $G$, starting from the worst
possible vertex. It has been studied extensively by computer scientists, due to its intrinsic appeal and its applications to designing universal traversal sequences \cite{AKLLR, Bridgland, Broder},
testing graph connectivity \cite{AKLLR, KR}, and protocol testing \cite{MP}; see \cite{aldous2} for an introduction to cover times.

Sophisticated methods to estimate the cover time have
been developed \cite{Feige-up, Feige-lower, Matthews, KKLV}. One of the most precise bounds was obtained by Kahn,
Kim, Lov{\'a}sz and Vu \cite{KKLV}. They gave polynomially
computable upper and lower bounds that differ by a factor of order
$(\log\log n)^2$. This breakthrough left several questions open:
\begin{description}
 \item[(i)] Can the bounds in \cite{KKLV} be represented by an explicit formula for concrete graphs of interest?
\item[(ii)] For such graphs, can the $(\log\log n)^2$ factor be removed?
\end{description}
In this work we improve the upper bound from \cite{KKLV} and show the resulting estimate is sharp up to a constant factor, and explicitly computable, for a large variety of graphs, in particular random graphs.

Let $G=(V,E)$ be a simple graph and write $\res(x,y)$ and $d(x,y)$
for the {\em effective resistance} and graph distance between two
vertices $x,y\in V$, respectively. See e.g. \cite{LP} or \cite{LPW} for definitions and properties of effective resistance. It is known that $\res(x,y) \leq
d(x,y)$ and that $\res(\cdot, \cdot)$ forms a metric on $G$. For
$x\in V$ and a real number $R>0$ we write $\ball(x, R)$ for the ball
of radius $R$ in the resistance metric, that is,
$$ \ball(x, R) = \{ v \in G : \res(x,v) \leq R \} \, .$$

\begin{theorem}\label{thm-cover}
Let $G=(V,E)$ be a finite graph with diameter $R$ in the resistance
metric.  For $i\in \N$, let
$A_i =A_i(G)$ be a set of minimal size such that
\be\label{def-a-cov} G \subset \bigcup_{ v \in A_i} \ball \big (v, {R \over 2^i}\big )\, ,\ee
and write $\alpha_i = 2^{-i}\log |A_i|$. Then
there exists a universal constant $C>0$ such that
\be \label{dudley} \tcov \le C\big(\mbox{$\sum_{i=1}^{\log_2 \log n}$} \sqrt{\alpha_i}\big)^2 R |E| \,.\ee
\end{theorem}

The right hand side is approximable up to constant factors in polynomial time, see Remark \ref{approx}. Theorem \ref{thm-cover} is a refinement on \cite{KKLV}, in which it is shown that
\be \max_i \alpha_i  R |E|\leq \tcov(G) \leq C (\log\log n)^2 \cdot \max_i \alpha_i R |E|.\ee
The lower bound is a variant of Matthews' estimate for cover times
\cite{Matthews}, and the upper
bound is the main contribution of \cite{KKLV}. We refine the methods
of \cite{KKLV} to deduce the stronger statement of Theorem
\ref{thm-cover} (clearly, $(\sum_{i=1}^{\log_2 \log
n}\sqrt{\alpha_i})^2 \leq (\log_2 \log n)^2 \max_i \alpha_i$). This
new bound turns out to be sharp in many concrete examples where we can show that
$\max_i \alpha_i$ and $(\sum_{i=1}^{\log_2 \log n}\sqrt{\alpha_i})^2$
are of the same order. Such examples are
presented in Theorems \ref{evolution} and \ref{covercrit} below. \\


Cooper and Frieze \cite{CF} studied the cover time of the largest
component of the Erd\H{o}s-R\'enyi \cite{ER59} random graph model
$G(n,p)$, that is, the random graph obtained from the complete graph
$K_n$ by retaining each edge with probability $p$ independently. It
is well known that if $p={c\over n}$ for some $c>1$, then the
largest connected component, $\C_1$, is of size about $xn$ with
probability tending to $1$, where $x=x(c)$ is the unique solution in
$(0,1)$ of $x=1-e^{-cx}$. Cooper and Frieze \cite{CF} established
the asymptotics for the cover time in this regime,
$$ \tcov(\C_1) \sim \varphi(c) n \log^2  n
\quad \mbox{ with } \quad \varphi(c) = {cx(2-x) \over 4(cx - \log
c)}$$ with probability tending to $1$ as $n \to \infty$.

Since $\varphi(c)$ tends to $1$ as $c\to 1$, one might be tempted to guess that $\tcov(\C_1)$ for $G(n,1/n)$ is of order $n \log^2 n$.  However, it is known~\cite{NP1} that the maximal hitting time between two vertices in $\C_1$ is typically of order $n$, so Matthews bound \cite{Matthews}   shows that $\tcov(\C_1)$ is at most $O(n \log n)$. In fact,  in $G(n,1/n)$ the largest component
  $\C_1$ is roughly of size $n^{2/3}$ \cite{ERDREN, bollo, Lucz},
and with probability uniformly bounded away from $0$ it is a tree.
Aldous \cite{Aldous} proved that a random tree on $k$ vertices has
cover time of order $k^{3/2}$ (see Theorem \ref{aldoustree} for a precise statement and an alternative proof). Combining these facts yields that
$\tcov(\C_1)$ in $G(n,{1 \over n})$ is of order $n$ with probability
uniformly bounded away from $0$. In the following theorem we show that this probability tends to $1$, and moreover, we
show how the order of the cover time continuously evolves from
the critical regime $c=1$ to the supercritical regime $c>1$.

\begin{theorem} \label{evolution} Let $\tcov(\C_1)$ denote the cover time of largest component of $G(n,p)$ and let $\lambda\in \mathbb{R}$ be fixed and $\eps(n)>0$ be a sequence such that $\eps(n)\to 0$ but $n^{1/3} \eps(n) \to \infty$. Then
\begin{enumerate}
\item[(a)] If $p={1 - \eps(n) \over n}$, then for any $\delta>0$ there exists $ B>0$ such that
    $$ \P\Big ( B^{-1} \eps^{-3} \log^{3/2}(\eps^3 n) \leq \tcov(\C_1) \leq B \eps^{-3} \log^{3/2}(\eps^3 n) \Big ) \geq 1-\delta  \, .$$
\item[(b)] If $p={1 + \lambda n^{-1/3} \over n}$, then then for any $\delta>0$ there exists $B>0$ such that
    $$ \P\Big ( B^{-1} n \leq \tcov(\C_1) \leq B n \Big ) \geq 1-\delta \, .$$
\item[(c)] There exists a constant $C>0$ such that if $p={1 + \eps(n) \over n}$, then
     $$ \P \Big ( C^{-1} n \log^2 (\eps^3 n) \leq \tcov(\C_1) \leq C n \log^2 (\eps^3 n)  \Big ) \to 1 \, .$$
\end{enumerate}
\end{theorem}

Theorem \ref{thm-cover} also allows us to prove sharp bounds on
cover time for critical percolation clusters, even when the
underlying graph is {\em not} the complete graph. Given a graph $G$
on $n$ vertices and $p\in[0,1]$, the random graph $G_p$ is obtained
from $G$ by retaining each edge with probability $p$ independently.
In the special case of $G = K_n$, this yields  the Erd\H{o}s-R\'enyi graph $G(n,p)$.
For a vertex $v\in G$ we write $\C(v)$ for the connected component
in $G_p$ containing $v$, and denote by $\C_1$ the largest connected
component of $G_p$. We are interested in critical percolation in which
$|\C_1|\approx n^{2/3}$. This occurs in numerous underlying graphs
$G$. A partial list of examples is:
\begin{enumerate}
\item The complete graph on $n$ vertices \cite{ERDREN, bollo, Lucz} with $p={1 + \Theta(n^{-1/3}) \over n}$,
\item A random $d$-regular graph \cite{NP2, Pittel} with $p={1 + \Theta(n^{-1/3}) \over d-1}$,
\item Expanders of high girth and degree $d$ \cite{Nachmias} with $p={1 + \Theta(n^{-1/3}) \over d-1}$,
\item The Hamming hypercube $\{0,1\}^m$ \cite{BCHSS} with $p$ satisfying $\E_p|\C(v)|=\Theta(n^{1/3})$,
\item Discrete tori $\Z_m^d$ for large but fixed dimension $d$ with $p=p_c(\Z^d)$ or $p$ satisfying $\E_p|\C(v)|=\Theta(n^{1/3})$ \cite{BCHSS, HH0, HH}.
\end{enumerate}
In all the examples above it is known that for any $\delta>0$ there
exists $B=B(\delta)>0$ such that
$$ \P_p\big ( B^{-1} n^{2/3} \leq |\C_1| \leq Bn^{2/3} \big ) \geq 1-\delta \, .$$
The following theorem is a generalization of part (b) of Theorem
\ref{evolution}, and states that in these cases $\tcov(\C_1)$ has
order $n$. This means that the cover time of the largest component
has the same order as the cover time of a random tree on the same
number of vertices. We note that unlike the $G(n,p)$ case, in
examples 4 and 5, the probability that the largest component is   a tree tends to zero as the volume grows,
 so the Aldous estimate \cite{Aldous} does not apply.
\begin{theorem} \label{covercrit} In examples $1-5$ above, we have that for any $\delta>0$ there exists $B=B(\delta)>0$ such that
$$ \P_p\big ( B^{-1} n \leq \tcov(\C_1) \leq Bn \big ) \geq 1-\delta \, .$$
\end{theorem}
\noindent In fact, in Section \ref{criticalsec} we provide a general
criterion for the conclusion of Theorem \ref{covercrit} to hold, which
applies to examples $1-5$, see Theorem \ref{criticalgraphs}. \\

\noindent {\bf Remark.} The {\em blanket\/} time $B$ is the expected
first time when the local times at all vertices are within a factor of $2$ from each other (the local time at a vertex $v$ is the number of visits to $v$ divided by the degree of $v$). This quantity was introduced by Winkler and Zuckerman \cite{WZ} (we use the definition of \cite{KKLV}) who conjectured that $B = O(\tcov)$ for any graph. The bounds in Theorems $1.1-1.3$ also apply to $B$ in place of $\tcov$. This will be clear from the proofs. \\

Finally, it is natural to guess that adding edges to a graph can only decrease the cover time. However, this is not the case, as  shown by the following example. Let $K_n^*$ be the graph obtained from $K_n$ (the complete graph  on $n$ vertices) by adding a new vertex $v$  and connecting it to one vertex of $K_n$. The cover time of $K_n^*$ is easily seen to be  $n^2$. On the other hand, if we replaces $K_n$ by $H_n$, a bounded degree expander on $n$ vertices, and construct $H_n^*$ by adding a new vertex $v$  and connecting it to one vertex of $H_n$, then the cover time of $H_n^*$ is of order $n \log n$. Since $H_n^*$ is a subgraph of $K_n^*$ on the same $n+1$ vertices, we conclude that adding an edge to a graph may increase the cover time.  The increase is at most by a constant factor:

\begin{proposition}\label{addedge} Let $G$ be a connected graph and let $u,v\in G$ be two vertices. Let   $G^+$ be the graph obtained from $G$ by adding the edge $\{u,v\}$ (if an edge connecting these two vertices already exists, then we add a multiple edge, and if $u=v$, then we add a loop). Then we have
$$ \tcov(G^+) \leq 4 \tcov(G) \, .$$
\end{proposition}

\section{Proof of Theorem \ref{thm-cover}}
Let $S_t$ be a simple random walk on $G$, and for an integer $t\geq
0$, define the {\em local time} $L^v_t$ of a vertex $v\in V$ by
\begin{equation}
  L^v_t \deq \frac{1}{d_v} \sum_{k=0}^t 1_{\{S_k = v\}}\,, \mbox{ for all
  } v\in V \mbox{ and } t\in \N\,,
  \end{equation}
where $d_v$ is the degree of vertex $v$. Furthermore, let $\tau^v_k
= \min \{t\in \N: L^v_t = k/d_v\}$ be the time of the $k$-th visit of the random walk to $v$. The following lemma of \cite{KKLV} implies that if the local time at a
vertex $u$ is large, then with high probability, the local time is also large at
vertices $v$ that are close to $u$ in the resistance metric.
\begin{lemma}\cite{KKLV}*{Lemma 5.2} \label{lem:ltbnd}
For all $u,v\in V$,  numbers $\lambda \geq 0$ and $t\in \N$  we have
\begin{equation*}
 \P_u\big( L^u_{\tau^u_t} - L^v_{\tau^u_{t}} \geq \lambda \big)
 \leq \exp \big(-\tfrac{\lambda^2}{4 t R_{\mathrm{eff}}(u, v)}\big)\,.
\end{equation*}
\end{lemma}

We use an idea of Kolmogorov \cite{stroock}*{page 91}. For all $i\geq 1$ and for each
$u \in A_{i}$, we can always select $v\in A_{i-1}$ such that
$R_{\bf{eff}}(u,v) \leq 2^{-(i-1)} R$ (see (\ref{def-a-cov})). Write $v = h(u)$. Set $\alpha'_i =
\alpha_i \vee 2^{-i/2}$ and define
\[\Psi = 128 (\mbox{$\sum_{i=1}^\infty $} \mbox{$\sqrt{\alpha'_i}$})^2, \mbox{
and } \beta_i = \frac{\sqrt{\alpha'_i}}{2\sum_{i=1}^\infty
\sqrt{\alpha'_i}}\mbox{ for all } i\in \N\,.\] For $i\in \N$, let
$t_i = (1 - \sum_{j=1}^{i} \beta_j) \Psi$,  and for $u\in A_{i}$
define $M_i(u)$ to be the difference of the local times of vertices
$h(u)$ and $u$ at time $\tau^{h(u)}_{t_{i-1} R}$, by
\[
M_i(u) = t_{i-1} R - L^u_{\tau^{h(u)}_{t_{i-1}R}} \, .\]
Lemma~\ref{lem:ltbnd} then gives that
 \[ \P( M_i(u) \geq \beta_i \Psi R) \le  \exp\Big(
-\frac{(\beta_i \Psi  R)^2 }{ 4 R 2^{-(i-1)} \cdot t_{i-1} R} \Big)
\leq e^{-  2^{i+1} \alpha'_i }\,.\] Define $\displaystyle M_i =
\max_{u\in A_{i}} M_i(u)$. Recalling the definition of $\alpha_i$
and $\alpha'_i$, we apply a union bound and get
\[\P(M_i \geq \beta_i \Psi R) \leq |A_{i}| e^{- 2^{i+1} \alpha'_i  } \leq \mathrm{e}^{- 2^i \alpha'_i} \,.\]
It follows that \be\label{hitall} \P \Big ( \mbox{$\bigcup_{i \geq
1}$} \{M_i \geq \beta_i \Psi R\} \Big ) \leq \sum _{i \geq 1}
\mathrm{e}^{ - 2^i \alpha'_i} \leq \sum_{i\geq 1}
\mathrm{e}^{-2^{i/2}} \leq \frac{2}{3} \, .\ee

Now, take $v \in V$ and write $\taucov$ for the cover time of the
random walk. Provided that the event $\mathcal{M} \deq
\bigcap_{i\geq 1}\{M_i \leq \beta_i \Psi R\}$ occurs, we have that
$L^u_{\tau^v_{\Psi R}} \geq (1- \sum_{j=1}^{i} \beta_i) \Psi R $ for
all $u\in A_i$ and hence $L^u_{\tau^v_{\Psi R}} \geq \Psi R/2$ by
the definition of $\beta_i$. In particular, on the event
$\mathcal{M}$ every vertex in the graph should have been visited at
least once. Combined with \eqref{hitall}, it follows that
$$\P_v(\taucov \geq \tau^v_{\Psi R}) \leq {2 \over 3} \, ,$$ and hence $\E_v \taucov \leq 3 \E_v \tau^v_{\Psi R}$. The expected return time to $v$ satisfies $\E_v T^+_v = {2 |E| \over d_v}$ whence
\begin{equation}\label{eq-cover-ball}\E_v \taucov \leq 3 \Psi R d_v \E_v T^+_v = 6 \Psi R |E| \, .\end{equation}
Since the above holds for all $v\in V$, we have $\tcov \leq 6 \Psi R
|E|$. Note that $|A_i| \leq n$ for all $i\in \N$ and hence
$\sum_{i\geq \log_2 \log n} \sqrt{\alpha_i} = O(1)$. Observing also
that $\alpha_i \leq \alpha'_i + 2^{-i/2}$, we get $\Psi \leq 256 (
(\sum_{i=1}^{\log_2 \log n} \sqrt{\alpha_i})^2 + 16)$. It completes
the proof of the theorem together with the fact that $\alpha_1 \geq
\frac{1}{2} \log 2$ (since $|A_1|$ has to be at least $2$).

\begin{remark}\label{approx}
Note that the sum $\sum_i \sqrt{\alpha_i}$ can be easily
approximated up to constant. To see this, one can use greedy
algorithm to find a maximal collection of centers $\tilde{A}_i$ such
that $\{B_{\mathbf{eff}}(v, 2^{-(i+1)} R) : v\in \tilde{A_i}\}$
forms a collection of disjoint balls. Thus, $|\tilde{A}_{i-1}| \leq
|A_{i}| \leq |\tilde{A}_i|$ and
\[\tfrac{1}{\sqrt{2}} \sum_i \sqrt{2^{-i}\log |\tilde{A}_i|} \leq \sum_{i} \sqrt{\alpha_i} \leq \sum_i \sqrt{2^{-i}\log|\tilde{A}_i|}\,.\]
\end{remark}

\section{Cover time of critical percolation clusters}\label{criticalsec}

We are interested in critical percolation in which $|\C_1|\approx
n^{2/3}$. This occurs in numerous underlying graphs $G$ as listed in
the introduction (examples $1-5$).
Recall the definition of $G_p$, and write $d_{G_p}(x,y)$ for the length of the shortest path between $x$
and $y$ in $G_p$, or $\infty$ if there is no such path. We call $d$
the {\em intrinsic metric} on $G_p$. Define the random sets
\begin{align*}
\bcr_p(x,r;G) = \{ u : d_{G_{p}}(x,u) \leq r \} \, , \quad \partial
\bcr_p(x,r;G) = \{ u : d_{G_{p}}(x,u) = r \} \, ,
\end{align*}
and the event
$H_p(x,r;G) = \Big \{ \partial B_p(x,r;G) \neq \emptyset \Big \} \, . $
 Finally, define
$$\Gamma_p(x,r; G)=\sup_{G' \subset G } \P (H_p(x,r;G')) \, ,
$$
where $\P$ here is the percolation probability measure over subgraphs of $G'$. The reason for taking a supremum in the definition of $\Gamma_p$ is that the event $H_p(x,r; G)$ is {\em not} monotone with respect to edge addition (indeed, adding an edge can potentially shorten a shortest path and make $\partial
\bcr_p(x,r;G)$ empty even if it were not empty before). The quantity $\Gamma_p$ is called the intrinsic metric {\em arm exponents} and was introduced in \cite{NP1}, see Theorem $2.1$ of that paper for further details there.

\begin{theorem}\label{criticalgraphs} Let $G=(V,E)$ be a graph and let $p \in [0,1]$. Suppose that for some constants $c_1, c_2>0$ and all vertices $x \in V$ the following two conditions are satisfied:
$$
(i) \,\, \E |\bcr_p(x,r;G)| \leq c_1 r \, , \qquad (ii) \,\, \Gamma_p(x,r; G) \leq c_2/r \, .
$$
Then for any $\beta,\delta>0$ there exists $B > 0$ such that
$$ \P \Big ( \exists \C \mbox{ with } |\C|\geq \beta n^{2/3} \mbox{ and } \tcov(\C) \not \in [B^{-1} n, Bn] \Big ) \leq \delta \, .$$
\end{theorem}
\begin{proof}
The fact that there exists $B>0$ such that
$$ \P \Big ( \exists \C \mbox{ with } |\C|\geq \beta n^{2/3} \mbox{ and } \tcov(\C) \leq B^{-1} n \Big ) \leq \delta/2 \, ,$$
follows immediately from the corresponding lower bound on the maximal hitting time, see part (c.2) of Theorem $2.1$ of \cite{NP1} and Lemma $4.1$ in that paper. Also from \cite{NP1} we have that for any
$\beta, \delta'>0$ there exists $D = D(\beta, \delta')>0$ such
that \be\label{diambound} \P \Big ( |\C(v)|\geq \beta n^{2/3} \mbox{
and } \diam(\C(v)) \not \in [D^{-1} n^{1/3}, Dn^{1/3}] \Big ) \leq
\delta' n^{-1/3}  \, .\ee
To see this, combine (3.1) and (3.3) of \cite{NP1}. Denote $\diameff(\C(v))$ for the diameter of
$\C(v)$ according to the resistance metric. We first show that with
high probability components of size $n^{2/3}$ have $\diameff$ of
order $n^{1/3}$. Indeed, the upper bound follows immediately from
Theorem 2.1 of \cite{NP1} and the fact that $\res(x,y)\leq d(x,y)$.
For the lower bound, we use Proposition 5.6 of \cite{NP1}, the
Nash-Williams inequality and (\ref{diambound}) to deduce that for
large enough $D = D(\beta, \delta')>0$ we have
\begin{equation}\label{eq-resistance-diam} \P \big ( |\C(v)| \geq \beta n^{2/3} \mbox{ and }
\diameff(\C(v)) \leq D^{-1}n^{1/3} \big ) \leq \delta' n^{-1/3} \,
.\end{equation}

We now proceed to construct covering sets of $G$ on different
scales. Fix an integer $i \geq 0$ and we define a sequence of radii
$\{r_j\}_{j \leq 2D^2 2^i}$ which have the following properties:
\begin{eqnarray*} (a) \,\,r_0 = 0 \, , \quad (b)\,\, {(j-1/2)n^{1/3} \over 2 D 2^i} \leq r_j \leq {jn^{1/3} \over 2 D 2^i} \, , \quad (c) \,\, \E \partial \bcr_p(v,r_j;G)  \leq 4D^2 c_1 2^i \, .\end{eqnarray*}
This is possible by condition (i) of the theorem, which implies that
for each $j \leq 2D^2 2^i$
$$ \sum _{\ell={(j-1/2)n^{1/3}/(2D 2^{i}) }}^{{jn^{1/3}/(2D 2^{i})}} \E \partial \bcr_p(v,\ell;G) \leq c_1D n^{1/3}\, ,$$
and so there must exists $\ell \in [(j-1/2)n^{1/3}/(2D 2^{i}),
jn^{1/3}/(2D 2^{i})]$ such that $r_j=\ell$ satisfies condition (c).
Given such radii $\{r_j\}$ we say that a vertex $u \in \partial
\bcr_p(v,r_j;G)$ is {\em $i$-good} if there exist a path between
$u$ and $\partial \bcr_p(v,r_{j+1};G)$ which does not go through
$\bcr_p(v,r_j;G)$. We now construct a sequence of sets $\{A'_i\}$ which will serve as a covering. Define
$$ A'_i = \bigcup_{j \leq 2D^2 2^i} \big \{ u \in \partial \bcr_p(v,r_j;G) \, : \, u \mbox{ is } \mbox{$i$-good} \big \} \, .$$
Observe that if $\diam(\C(v))\leq Dn^{1/3}$ then we have that
$$\C(v) \subset \bigcup_{u \in A'_i} \bcr_p \big (u, {2^{-i} D^{-1}n^{1/3}}; G \big )\, .$$ Furthermore, if in addition $R=\diameff(\C(v))\geq D^{-1} n^{1/3}$, then we
have that
$\C(v) \subset \bigcup_{u \in A'_i} \bcr_p \big (u,  2^{-i} R; G\big ) $. Given these two events and the fact that $\bcr_p \big (u,r; G)
\subset \ball(u,r; \C(v))$, we deduce that
\[\C(v) \subset \bigcup_{u \in A'_i}  B_{\bf{eff}}\big (u,  2^{-i} R ; \C(v) \big )
\,,\] and therefore $|A_i| = |A_i(\C(v))| \leq |A'_i|$ for all $i\in
\N$ (see (\ref{def-a-cov})). By (\ref{diambound}) and (\ref{eq-resistance-diam}), we get that
\begin{equation}\label{eq-A-A'}
\P(|\C(v)| \geq \beta n^{2/3}, \exists i \in \N: |A'_i|< |A_i|)
\leq 2\delta' n^{-1/3}\,.
\end{equation}

Now, by condition (ii) of our theorem and our construction of
$\{r_j\}$ we get that
$$ \E|A'_i| \leq \sum _{j \leq 2D^2 2^i} \E \partial \bcr_p(v,r_j;G) \cdot {4D c_2 2^i \over n^{1/3}} \leq 16 D^3 c_1 c_2 2^{2i} n^{-1/3} \, .$$
So we can choose a large integer $m=m(c_1, c_2, D, \delta')$ such
that
\be\label{fromhere} \sum _{i=1}^\infty \P\big(|A'_i| \geq  \mathrm{e}^{m \cdot 2^{i/2}}\big) \leq  \sum_{i=1}^\infty  \frac{ 16 D^3c_1 c_2 2^{2i} n^{-1/3}}{ \mathrm{e}^{m \cdot 2^{i/2}}} \leq  \delta'  n^{-1/3}\, .\ee
Recalling that (see Theorem~\ref{thm-cover}) $\alpha_i  = 2^{-i}\log
|A_i|$ and combining the above estimate with \eqref{eq-A-A'}, we
obtain that
\begin{align}\label{eq-alpha-s}\P\big(|\C(v)| \geq \beta n^{2/3}, \mbox{$\sum_{i=1}^\infty$} \sqrt{\alpha_i} \geq
4m \big) \leq& \P\big(|\C(v)| \geq \beta n^{2/3}, \exists i\in\N :
|A'_i| < |A_i| \big)\nonumber \\
& + \sum _{i=1}^\infty \P\big(|A'_i| \geq \mathrm{e}^{m \cdot
2^{i/2}}\big) \leq 3 \delta' n ^{-1/3}\,.\end{align}
We say that $\C(v)$ is {\em bad} if $|\C(v)| \geq
\beta n^{2/3}$ and one of the following holds:
\begin{itemize}
\item $\sum_{i=1}^\infty \sqrt{\alpha_i} \geq 4m$, or
\item $\diameff(\C(v)) \geq D n ^{1/3}$, or 
\item $|E(\C(v))| \geq Dn^{2/3}$. 
\end{itemize}
By \eqref{eq-alpha-s} and Theorem $2.1$ of \cite{NP1} we learn that we can choose $D$ large enough so that the probability that $\C(v)$ is bad is at most $5\delta'n^{-1/3}$, whence $\E X \leq 5\delta' n^{2/3}$. Note that if there exists $v$ such that $\C(v)$ is bad, then $X \geq \beta n^{2/3}$.  By
Theorem \ref{thm-cover} we learn that there exists some large
constant $B=B(D,m)$ such that if $|\C(v)|\geq \beta n^{2/3}$ and $\tcov(\C(v)) \geq Bn$, then $\C(v)$ is bad (taking $B=16 C m^2 D^2$, where $C$ is the constant of Theorem \ref{thm-cover} suffices). Hence, by Markov's inequality
$$ \P \Big ( \exists \C \mbox{ with } |\C|\geq \beta n^{2/3} \mbox{ and } \tcov(\C) \geq Bn \Big ) \leq \P( X \geq \beta n^{2/3}) \leq 5\delta'/\beta \, ,$$
which concludes the proof of the theorem by setting $\delta' =
\delta/(10\beta)$. \qed
\end{proof}

\noindent {\bf Proof of Theorem \ref{covercrit}.} We only need to show that the conditions of Theorem \ref{criticalgraphs} holds in examples $1-5$. Indeed, it is shown in \cite{NP1} that the conditions hold for examples $1-3$, and in \cite{KN1} and \cite{KN2} it is shown for examples $4-5$. In \cite{HH, HH0} it is shown for example $5$ that at $p=p_c(\Z^d)$ the largest cluster size is of order $n^{2/3}$. \qed


We will require the following result of Aldous \cite{Aldous}. For the reader's convenience we provide a simpler proof of this theorem based on Theorem \ref{thm-cover}.
\begin{theorem} \label{aldoustree} Let $T$ be a Galton-Watson tree with progeny mean $1$ and variance $\sigma^2<\infty$.
Then for any $\delta>0$ there exists $A=A(\delta, \sigma^2)>0$ such that
$$ \P \big ( \tcov(T) \not \in [A^{-1} k^{3/2}, A k^{3/2}] \, \big | \, |T| \in [k,2k] \big )\leq \delta \, .$$
\end{theorem}
\begin{proof} This is very similar to the proof of Theorem \ref{criticalgraphs}. Firstly, we claim that there exists $D>0$ such that
$$ \P \big ( \diam(T) \not \in [D^{-1} k^{1/2}, Dk^{1/2}] \, , |T| \in [k,2k] \big ) \leq k^{-1/2} \delta /2 \, .$$
Indeed, it is a classical fact \cite{KSN} that $\P(\diam(T) \geq Dk^{1/2}) = O(D^{-1}k^{-1/2})$. Furthermore, the expected number of particles in $T$ up to level $D^{-1} k^{1/2}$ is precisely $D^{-1} k^{1/2}$, and the event $\{\diam(T)\leq D^{-1}k^{1/2} \, , |T| \geq k\}$ implies that this quantity is at least $k$. Hence by Markov's inequality we have that $\P(\diam(T) \leq D^{-1} k^{1/2}, |T| \geq k) \leq D^{-1} k^{-1/2}$.

Now, for each $i$ we define $r_j = j 2^{-i-1} D^{-1} \sqrt{k}$ for $j=0, \ldots, 2^{i+1} D^2$ and define $A_i'$ to be the set of particles at level $r_j$ which survive up to level $r_{j+1}$. As in the proof of Theorem \ref{criticalgraphs}, if $\diam(T) \in [D^{-1} k^{1/2}, Dk^{1/2}]$, then
$$ T \subset \bigcup_{u \in A'_i}  B_{\bf{eff}}\big (u,  2^{-i} R ; T \big ) \, ,$$
where $R$ is the diameter of $T$ with respect to the resistance metric. Now, for each $j$ the expected number of particles in level $r_j$ is precisely $1$ and for each, the probability of surviving up to level $r_{j+1}$ is of order $(r_{j+1}-r_j)^{-1}$ (see \cite{KSN} again), hence $\E |A_i'| \leq C 2^{2i+2} D^3 k^{-1/2}$  and the proof continues as in (\ref{fromhere}) to show using Theorem \ref{thm-cover} that there exists $A$ such that
$$ \P \big ( \tcov(T) \geq A k^{3/2} \, , |T| \in [k,2k] \big ) \leq k^{-1/2} \delta/2 \, .$$
Let $L$ be the offspring random variable of $T$. We have that $|T|$ is distributed as the first hitting time of $0$ of a random walk starting $1$ with increments distributed as $L-1$ (see exercise $5.26$ of \cite{LP}). We use this and Theorem 1a of chapter XII.7 in \cite{Feller2} to deduce that
$$ \P(|T|\in[k,2k]) = (1+o(1))Ck^{1/2} \, ,$$ for some constant $C>0$.
This gives the required upper bound on the cover time. The corresponding lower bound follows immediately from the lower bound on the maximal hitting time, which we obtain via the $\sqrt{k}$ lower bound on the diameter of $T$ together with commute time identity. \qed \\ \end{proof}

\noindent {\bf Proof of part (a) and (b) of Theorem \ref{evolution}.}
Part (b) of the theorem follows immediately from Theorem \ref{criticalgraphs}, so we are only left to prove part (a). In this case it is known that the largest cluster is a uniform random tree of order $\eps^{-2} \log(\eps^3 n)$ (see \cite{Jansonbook}). It is a classical fact (see chapter $2.2$ of \cite{Kolchin}) that a uniform random tree of size $k$ is distributed as a Poisson($1$) Galton-Watson tree $T$ conditioned on $|T|=k$. Hence the following statement concludes the proof: let $T$ be a Poisson($1$) Galton-Watson tree, then for any $\delta>0$ there exists $A>0$ such that
\be \label{aldouswinkler} \P \big ( \tcov(T) \not \in [A^{-1} k^{3/2}, A k^{3/2}] \, \big | \, |T| = k \big ) \leq \delta \, .\ee
Note that this assertion does not immediately follow from Theorem \ref{aldoustree}. To fill in the gap, we will infer from a result Luczak and Winkler \cite{LW}, that there exists a coupling between a random tree $T_k$ of size $k$ and a random tree $T_{k+1}$ of size $k+1$ such that $T_k \subset T_{k+1}$. This together with Theorem \ref{aldoustree} shows the the upper bound on the cover time of (\ref{aldouswinkler}) and concludes the proof (the lower bound on the cover time is easier and follows, as in the remark above, by the easy lower bound on the maximal hitting time).

To see that such a coupling exists write $T^{(d)}_k$ for a Bin($d,1/d$) Galton-Watson tree conditioned on being of size $k$. Theorem $4.1$ in \cite{LW} shows that there exists a coupling between $T^{(d)}_k$ and $T^{(d)}_{k+1}$ such that $T^{(d)}_k \subset T^{(d)}_{k+1}$. Now, for any fixed $k$ we may take $d \to \infty$ and we get the required coupling between Poisson($1$) Galton-Watson trees. This concludes our coupling since the latter trees are uniform random trees. \qed

\section{Cover time for mildly supercritical Erd\H{o}s-R\'enyi graph}

In this section, we prove Part (c) of Theorem~\ref{evolution}, which
incorporates the order of the cover time for the largest component
of Erd\H{o}s-R\'enyi graph $G(n, p)$ with $p = \frac{1+\eps}{n}$,
where $\eps = o(1)$ and $\eps^3 n \to \infty$. Our proof makes
use of the following structure result of \cite{DKLP1}.

\begin{theorem}\cite{DKLP1}\label{mainthm-struct-gen}
Let $\GC$ be the largest component of $G(n,p)$ for $p = \frac{1 +
\eps}{n}$, where $\eps^3 n\to \infty$ and $\eps\to 0$. Let
$\mu<1$ denote the conjugate of $1+\eps$, that is,
$\mu\mathrm{e}^{-\mu} = (1+\eps) \mathrm{e}^{-(1+\eps)}$. Then
$\GC$ is contiguous to the model $\tGC$ constructed in the following
3 steps:
\begin{enumerate}
  \item[(a)]\label{item-struct-gen-degrees} Let $\Lambda\sim \mathcal{N}\left(1+\eps - \mu, \frac1{\eps n}\right)$ and assign
  i.i.d.\  variables $D_u \sim \mathrm{Poisson}(\Lambda)$ ($u \in [n]$) to the vertices, conditioned that $\sum D_u \one_{D_u\geq 3}$ is even.
 Let $N_k = \#\{u : D_u = k\}$ and $N= \sum_{k\geq 3}N_k$. Select a random graph $\K$ on $N$ vertices, uniformly among all graphs with $N_k$ vertices of degree $k$ for $k\geq 3$.
  \item[(b)]\label{item-struct-gen-edges} Replace the edges of $\K$ by paths of lengths i.i.d.\ $\mathrm{Geom}(1-\mu)$. 
  \item[(c)]\label{item-struct-gen-bushes} Attach an independent $\mathrm{Poisson}(\mu)$-Galton-Watson tree (PGW tree in what follows) to each vertex.
\end{enumerate}
That is, $\P(\tGC \in \mathcal{A}) \to 0$ implies $\P(\GC \in
\mathcal{A}) \to 0$ for any set of graphs $\mathcal{A}$.
\end{theorem}
By the above theorem, it suffices to analyze the cover time of
$\tGC$. In what follows, we will repeatedly use some known facts
about $\tGC$ and one can see \cite{DKLP1, DKLP2} for references.

\subsection{Lower bound}
We first show that w.h.p.~there are $(\eps^3 n)^{1/4}$ attached
trees, as in part (c) of the construction of $\tGC$, of height at least $\frac{1}{2}\eps^{-1} \log (\eps^3 n)$. To this end,  note that the height $H$ of a PGW($\mu$) tree satisfies the following for some constant $c>0$ (see, e.g.,
\cite{DKLP2}*{Lemma 4.2})
\be\label{treebound}\P \big (H \geq \tfrac{1}{2} \eps^{-1} \log(\eps^3 n) \big) \geq c \eps (\eps^3 n)^{-1/2 + o(1)}\,, \ee
where we used the fact that  $\mu = (1 - (1+o(1)) \eps)$. It is an immediate consequence of parts (a) and (b) of the construction of $\tGC$ that w.h.p.~there are $(2+o(1)) \eps^2 n$ i.i.d.\ attached PGW($\mu$)
trees. Hence, by (\ref{treebound}), we learn that with high probability there are at least $(\eps^3 n)^{1/4}$ PGW trees of height at least $\tfrac{1}{2} \eps^{-1} \log(\eps^3 n)$. Now, take exactly one leaf in the bottom level from each of these trees and denote by $B$ the set of these leaves.  We will use the following lemma (see, e.g., \cite{Tetali}, and also see \cite{LP}*{Proposition 2.19}) to bound the hitting time between vertices in $B$.
\begin{lemma}\label{lem-network-hitting-time}
Given a finite network with a vertex $v$ and a subset of
vertices $Z$ such that $v\not \in Z$. Let $vol(\cdot)$ be the voltage when a unit current
flows from $v$ to $Z$ and $vol(Z) = 0$. Then we have that $\E_v[\tau_Z]
=\sum_{x \in V} c(x) vol(x)$, where $c(x) = \sum_{x \sim y} c(x, y)$
and $c(x, y)$ is the conductance between $(x, y)$.
\end{lemma}
In our setting, $c(x, y) = 1$ if $(x, y)$ is an edge of
$\tGC$, and otherwise $c(x, y) = 0$. Let $u, v\in B$, and let $T(v)$
be the attached PGW tree that contains $v$. It is clear that for all
$w\not\in T(v)$ the effective resistance between $w$ and $v$
satisfies $R_{\bf{eff}}(w, v) \geq  (2\eps)^{-1} \log(\eps^3 n)$.
Now, if a unit current flows from $u$ to $v$
and the voltage at $v$ is set to be $0$, we can then deduce that the
voltage at vertex $w$ is at least $(2\eps)^{-1} \log(\eps^3 n)$,
for all $w\not\in T(v)$. Note that w.h.p.\ simultaneously for all
$v\in B$ we have $|\tGC\setminus T(v)| = (2+o(1))\eps n$ (see
\cite{DKLP1}) and we then assume this. Lemma
~\ref{lem-network-hitting-time} then yields that for all $u, v\in B$
\[\E_{u} \tau_v \geq  (2\eps)^{-1} \log(\eps^3 n) (2+o(1))
\eps n = (1+o(1))  n \log (\eps^3 n)\,.\] At this point, an
application of the Matthews lower bound \cite{Matthews} (see also, e.g.,
\cite{LPW}) stating that for any subset $A \subset G$ we have
$\tcov(G) \geq \log|A| \min_{u,v\in A} \E_u\tau_v$,
completes the proof of the lower bound. \qed

\subsection{Upper bound}
In this section we establish the upper bound on the cover time. In
light of Theorem~\ref{thm-cover}, it suffices to show that w.h.p.\
for $\tGC$ we have that $|A_i| \leq (\eps^3 n)^{ 2i} $
simultaneously for all $i\geq 1$. Let $R$ be the diameter of $\tGC$
in resistance metric. As shown in \cite{DKLP2}, with high
probability the diameter in graph metric is $(3+o(1))\eps^{-1}
\log (\eps^3 n)$ and also the two highest attached trees have
height $(1+o(1))\eps^{-1} \log (\eps^3 n)$ each. It implies that
$(2+o(1))\eps^{-1} \log (\eps^3 n) \leq R \leq (3+o(1))
\eps^{-1} \log (\eps^3 n)$ w.h.p., and we assume this in what
follows.

Fix $i\in \mathbb{N}$, we now construct $A'_i$ such that balls of radius
$2^{-i}R$ around vertices in $A'_i$ form a covering of
$\tilde{C}_1$. We first cover the 2-core $\mathcal{H}$ of $\tGC$ by
balls of radius $2^{-(i+1)} R$. To this end, consider the disjoint
balls of radius $2^{-(i+2)}R$ that can be packed in $\mathcal{H}$.
Take such a maximal packing and denote by $A'_{i, 1}$ the set of
these centers. Since the packing is maximal, we have that
\[\mathcal{H} \subseteq \bigcup_{v\in A'_{i,1}} \ball(v, 2^{-(i+1)} R)\,.\]
Since $\res(x,y)\leq d(x,y)$, it follows that $|\ball(v, 2^{-(i+2)} R) \cap \mathcal{H}| \geq 2^{-(i+2)}R$ for all $v\in A'_{i,1}$. Therefore, since the balls $\ball(v, 2^{-(i+2)} R)$ for $v\in A'_{i,1}$ are disjoint, we conclude that $|A'_{i,1}| \leq 4 \cdot 2^i|\mathcal{H}|/R$.

We now turn to cover the attached trees. For a rooted tree $T$, let
$H(T)$ be the height of $T$. For $v\in T$, denote by $T_v$ the
subtree of $T$ rooted at $v$ that contains all the descendants of
$v$. Also, denote by $L_k$ the vertices in level $k 2^{-(i+1)} R$ of
$T$. Define \[F_T \deq \cup_{k=1}^\infty\{v\in L_k: H(T_v) \geq
2^{-(i+1)} R\}\,. \]  Let $\mathcal{T}$ be the collection of
attached PGW trees in $\tGC$ and let $A'_{i,2} = \cup_{T \in
\mathcal{T}} F_T$. Defining $A'_i = A'_{i,1} \cup A'_{i,2}$, we
deduce from the definition that
$\tGC \subseteq \bigcup_{v\in A'_i} \ball(v, 2^{-i}R)$.
It remains to bound $|A'_{i,2}|$. Using \cite{DKLP2}*{Lemma 4.2}
again, we obtain that for a PGW$(\mu)$ tree $T$ and some absolute
constant $C$,
\be\label{estimate}\P(H(T) \geq 2^{-(i+1)} R) \leq \begin{cases}
 C \eps\,, &\mbox{ if } 2^{i} \leq \log(\eps^3 n)\,,\\
\frac{C}{2^{-(i+1)}R}\,, & \mbox{ if } 2^{i} \geq \log(\eps^3
n)\,.
\end{cases}
\ee
Also, it is immediate that $\E[|L_k|] = \mu^{k2^{-(i+1)} R}$. Furthermore, by the Markov property, given $|L_k|$ the set $\{T_v: v\in L_k\}$ is distributed as $|L_k|$ independent copies of $T$. By this and (\ref{estimate}) we get that for some absolute constant $C>0$
\begin{align*}\E[F_T]& = \sum_{k\geq 1} \E |\{v \in L_k: H(T_v) \geq 2^{-(i+1)}R\}| = \sum_{k\geq 1} \E [|L_k|] \P(H(T_v \geq 2^{-(i+1)} R)) \\
&\leq \begin{cases} \sum_{k\geq 1} \mu^{k 2^{-i} R /2} \cdot C
\eps \leq C^2
\eps\,, &\mbox{ if } 2^{i} \leq \log(\eps^3 n)\,,\\
\sum_{k\geq 1} \mu^{k 2^{-(i+1)}R} \cdot \frac{C}{2^{-(i+1)}R} \leq
C^2 2^{2i} /R & \mbox{ if } 2^{i} \geq \log(\eps^3 n)\,.
\end{cases}
\end{align*}
Hence, we can always get $\E[F_T] \leq C^2 \eps 2^{2i}$.
Furthermore,  it is known that $|\mathcal{H}|= (2+o(1))\eps^2 n$
with high probability so we may assume this. By Markov's
inequality and the fact that $|A'_{i,1}| \leq 4 \cdot
2^i|\mathcal{H}| /R = o((\eps^3 n)^{ 2i})$ we have that
\begin{align*}\P(|A'_i| \geq (\eps^3 n)^{ 2i}) &= \P(|A'_{i,2}| \geq (\eps^3 n)^{ 2i} -
|A'_{i,1}|) \leq \frac{\E[|A'_{i,2}|]}{(\eps^3 n)^{ 2i}-
|A'_{i,1}|} = \frac{|\mathcal{H}| \E[F_T]}{(1+o(1))(\eps^3 n)^{ 2i}}\\
& \leq (2+o(1)) C^2 \eps^3 n 2^{2i} (\eps^3 n)^{-2i} \leq
o(1)C^2 (\eps^3 n /8)^{-2(i-1)}\,.\end{align*} A simple union
bound gives that with high probability $|A'_i| \leq (\eps^3 n)^{
2i}$ simultaneously for all $i\geq 1$. Recalling the facts that
$|E(\tGC)| = (2+o(1)) \eps n$ and $R \leq 3+o(1) \eps^{-1} \log
(\eps^3 n)$, we conclude the proof of the upper bound by an
application of Theorem~\ref{thm-cover}. \qed

\section{Proof of Proposition \ref{addedge}}
We may assume that $|E(G)| \geq 2$.
Let $\pi$ be the stationary distribution of $G$ and let
$\{S^+_t\}_{t \geq 0}$ be a random walk on $G^+$ starting from the
initial distribution $\pi$ (note that $\pi$ is {\em not} the
stationary distribution for $G^+$).
Let $\tau_0 = \tau'_0 = 0$ and for all $i\geq 1$ define
\[\tau_{i}\deq \min \big\{t \geq \tau'_{i-1}: \{S^+_{t}, S^+_{t+1}\} = \{u, v\}\big\}\,, \, X_{i} \deq S^+_{\tau_i}\,, \mbox{ and } \tau'_i \deq \min \{t > \tau_i: S^+_t = X_i\}\,.\]
Write $T_i = \{t: \tau_i < t\leq \tau'_i\}$ and for all $t\in \N$
further define
\[\Phi(t) = \min\{k : |[0, k]\setminus \cup_{i=1}^\infty T_i| = t\}\,.\]
Now let $S_t = S^+_{\Phi(t)}$. We first claim that $S_t$ is a simple
random walk on the graph $G$. In order to see that, one just need to
note that $S_t$ is obtained from $S^+_t$ by omitting all the
excursions started with traveling through the edge $(u, v)$. Let
$\taucov$ be the first time when $S_t$ visits every vertex of $G$
and it then remains to bound $\E[\Phi(\taucov)]$.

 To this end, it is more convenient to consider the first time $\taucov^*$ when $S_t$ visits every vertex of $G$ and returns to the starting point. We wish to bound
the number of steps spent on the above defined excursions before
$\taucov^*$. Define \begin{align*}L_u(\taucov^*)& = \big | \big \{t\leq
\taucov^*: S_t = u \big\}\big|\,\, \mbox{ and } \,\,L_v(\taucov^*) = \big | \big \{t\leq
\taucov^*: S_t = v
\big \}\big| \,,\\
N_u(\taucov^*) &= \big | \big \{i: T_i \subseteq [0, \Phi(\taucov^*)],
X_{i} = u \big\} \big| \,\, \mbox { and }\,\, N_v(\taucov^*) = \big | \big\{i: T_i
\subseteq [0, \Phi(\taucov^*)], X_{i} = v \big \} \big| \,.\end{align*}

Note that every time when $S_t = u$, the corresponding random walk
$S^+_{\Phi(t)}$ is also at $u$ and has chance $\frac{1}{d_u+1}$ to
travel to $v$ and thus starts an excursion, and moreover, once
started the number of excursions has law $\mathrm{Geom}(1/(d_u+1))$
independent of $\{S_t\}$. Therefore, we have
$$N_u(\taucov^*) = \sum_{i=1}^{L_u(\taucov^*)} Y_i Z_i \, ,$$
where $\{(Y_i, Z_i)\}$ are independent and $Y_i \sim
\mathrm{Ber}(1/(d_u+1))$ and $Z_i \sim \mathrm{Geom}(1/(d_u + 1))$.
Thus, $\E[N_u(\taucov^*)] = \frac{1}{d_u}\E[L_u(\taucov^*)]$. By
\cite{AF}*{Chapter 2, Proposition 3}, we know that
$\E[L_u(\taucov^*)] = \frac{d_u}{2 |E(G)|} \E[\taucov^*]$ and
therefore $\E[N_u(\taucov^*)] = \frac{1}{2|E(G)|} \E [\taucov^*]$.
Suppose $X_{i} = u$, each $T_i$ is distributed as $1 +
\tau^+_u$ where $\tau^+_u$ is the hitting time of $S^+_t$ to $u$
started at $v$. Observing that $\{|T_i|\}$ are independent of
$N_u(\taucov^*)$, we can then obtain that
\[\mathrm{Exc}(u)\deq\E[|\cup_i\{T_i \subseteq [0, \Phi(\taucov^*)]: X_{i} = u\}|] = \frac{1}{2|E(G)|} \E [\taucov^*] (1 + \E_v[\tau^+_u])\,.\]
In the same manner, we derive that
\[\mathrm{Exc}(v)\deq\E[|\cup_i\{T_i \subseteq [0, \Phi(\taucov^*)]: X_{i} =v\}|] = \frac{1}{2|E(G)|} \E [\taucov^*] (1 + \E_u[\tau^+_v])\,.\]

Note that $\E_v [ \tau^+_u] + \E_u [\tau^+_v]$ is the expected
commute time between $u$ and $v$ and hence by commute identity
\cite{CRRST}, we have $\E_v [ \tau^+_u] + \E_u [\tau^+_v] = 2
|E(G^+)| R^+(u, v)$, where $R^+(u, v)$ is the resistance between $u$
and $v$ in $G^+$. Since $G$ is connected, we get $R^+(u, v) \leq
\frac{|E(G)|}{|E(G)| + 1}$. Altogether,
\[\tcov(G^+) = \E[\Phi(\taucov)] \leq \tcov(G) + \mathrm{Exc}(u) + \mathrm{Exc}(v) \leq 3 \tcov(G) + \frac{2}{|E(G)|} \tcov(G) \leq 4 \tcov(G)\,,\]
where we used the inequality $\E[\taucov^*] \leq 2 \tcov$ and the
assumption that $|E(G)| \geq 2$. \qed

\begin{remark}
If $G^+$ is obtained from a connected graph $G$ by adding $k$ extra
edges, a similar argument gives that
\[\tcov(G^+) \leq \big(2k+1 + \tfrac{2k^2}{|E|}\big)\tcov(G)\,.\]
\end{remark}

\section{A concluding remark}

The bound (\ref{dudley}) is reminiscent of Dudley's entropy bound for Gaussian process \cite{Dudley}. Motivated by this, Ding, Lee and Peres \cite{DLP} show the link to Gaussian processes is much tighter. In particular, Talagrand's majorizing measures bound for Gaussian processes (see \cite{Talagrand}) can be used to estimate the cover time up to a multiplicative constant.

\end{document}